\documentclass[journal]{ieeeconf}

\usepackage{etex}

\IEEEoverridecommandlockouts                              

\usepackage{psfrag}
\usepackage{cite}
\usepackage{latexsym}
\usepackage{graphicx}
\usepackage{float}
\usepackage{colortbl}
\usepackage{fancyhdr}
\usepackage{subcaption}
\usepackage{hyperref}
\usepackage{enumerate}
\usepackage{multicol}
\usepackage{float}
\usepackage{mathtools, cuted}
\usepackage{graphicx,color}
\usepackage{amsmath,amssymb,amsfonts,booktabs, mathrsfs}
\usepackage{bbm,dsfont}
\usepackage[usenames,dvipsnames,svgnames,table]{xcolor}
\usepackage{blkarray}
\usepackage{tikz}
\usepackage{mhchem}
 \usepackage{ tikz-3dplot}  
 \usepackage{multido} 
 \usetikzlibrary{matrix,arrows,automata,positioning,shapes,calc,backgrounds,decorations.pathmorphing,shadows}
\usetikzlibrary{arrows}
\usetikzlibrary{calc}
\usepackage{pgfplots}
\usepgfplotslibrary{patchplots}

\newtheorem{theorem}{Theorem}

\newtheorem{lemma}[theorem]{Lemma}

\newtheorem{definition}{Definition}

\newtheorem{problem}{Problem}

\newcommand{\ds}{\displaystyle}

\newcommand{\enma}[1]   {\ensuremath{#1}}

\newcommand{\beq}{\begin{equation}}
\newcommand{\eeq}{\end{equation}}
\newcommand{\bseq}{\begin{subequations}}
\newcommand{\eseq}{\end{subequations}}
\newcommand{\beqn}{\begin{eqnarray}}
\newcommand{\eeqn}{\end{eqnarray}}
\newcommand{\ba}{\begin{array}}
\newcommand{\ea}{\end{array}}
\newcommand{\bct}{\begin{center}}
\newcommand{\ect}{\end{center}}
\newcommand{\btmz}{\begin{itemize}}
\newcommand{\etmz}{\end{itemize}}
\newcommand{\benum}{\begin{enumerate}}
\newcommand{\eenum}{\end{enumerate}}

\newcommand{\mc}{\mathcal}

\newcommand{\R}{{\mathbb R}}

\newcommand{\Z}{{\mathbb Z}}

\newcommand{\blockdiag} {\enma{\mathrm{blockdiag}}}

\newcommand{\be}{\begin{equation}}
\newcommand{\ee}{\end{equation}}

\newcommand{\cplxs}{ C\kern -.35em \rule{0.03 em}{.7 ex}~   }

\def\complex{\hbox{C\kern -.45em \rule{0.03 em}{1.5 ex}}~}

\newcommand{\bi}{\begin{itemize}}
\newcommand{\ei}{\end{itemize}}

\newtheorem{assumption}{Assumption}

\newcommand{\vsp}{\vspace{0.1cm}}

\graphicspath{{Figures/} }

\begin{document}

\title{\LARGE \bf Mutually Quadratically Invariant Information Structures in \\Two-Team Stochastic Dynamic Games}
\author{Marcello Colombino$^\dag$, Roy S.\ Smith$^\dag$, and  Tyler H.\ Summers$^\ddag$%
\thanks{$^\dag$M. Colombino and R. Smith are with the Automatic Control Laboratory, ETH Zurich, Switzerland. $^\ddag$T. Summers is with the Department of Mechanical Engineering, University of Texas at Dallas, USA.
E-mail addresses: \{\texttt{mcolombi}, \texttt{rsmith}\}\texttt{@control.ee.ethz.ch}, \texttt{tyler.summers@utdallas.edu}. This research is supported by the National Science Foundation under grant CNS-1566127 and partially supported by the Swiss National Science Foundation grant 2--773337--12.}%
}

\maketitle         
\thispagestyle{empty}
\pagestyle{empty}


\begin{abstract}
We formulate a two-team linear quadratic stochastic dynamic game featuring two opposing teams each with decentralized information structures. We introduce the concept of mutual quadratic invariance (MQI), which, analogously to quadratic invariance in (single team) decentralized control, defines a class of interacting information structures for the two teams under which optimal linear feedback control strategies are easy to compute. We show that, for zero-sum two-team dynamic games, structured state feedback Nash (saddle-point) equilibrium strategies can be computed from equivalent structured disturbance feedforward saddle point equilibrium strategies. However, for nonzero-sum games we show via a counterexample that a similar equivalence fails to hold. The results are illustrated with a simple yet rich numerical example that illustrates the importance of the information structure for dynamic games. 
\end{abstract}
\section{introduction}

Future cyber-physical systems (CPS) will feature cooperative networks of autonomous decision making agents equipped with embedded sensing, computation, communication, and actuation capabilities. These capabilities promise to significantly enhance performance, but also render the network vulnerable by increasing the number of access and influence points available to attackers. ``Red team-blue team'' scenarios, in which a defending team seeks to operate the network efficiently and securely while the attacking team seeks to disrupt network operation, have been used to qualitatively assess and improve security in military and intelligence organizations, but have not received formal mathematical analysis in a CPS context. Here we will study some fundamental properties in two-team stochastic dynamic games in cyber-physical networks, with a focus on interactions of the information structures of each team.

Dynamic game theory \cite{basar1995dynamic} offers a general framework for the study of optimal decision making in stochastic and non-cooperative environments. The theory can be viewed as a marriage of game theory \cite{von1944theory}, with a focus on interactions of multiple decision making agents, 
and optimal control theory \cite{bellman1952theory,pontryagin1959optimal},
with a focus on dynamics and feedback. The main elements of dynamic game theory are (1) a dynamical system along with a set of agents whose actions influence the state evolution of the system, (2) an objective function to be optimized associated with each agent, and (3) an information structure that specifies information sets for each agent, i.e., who knows what and when. 

Several special classes of dynamic games have been extensively studied. In team decision theory all agents cooperate to optimize the same objective function. Static team theory traces back to the seminal work of Marschak and Radner \cite{marschak1955elements,radner1962team,marschak1972economic}. Decentralized control theory has developed from team decision theory and control theory and introduces dynamics. The presence of dynamics makes available information depend on the actions of agents and significantly complicates the problem. Dynamic aspects were studied in important early work by Witsenhausen \cite{witsenhausen1968counterexample,witsenhausen1971separation,witsenhausen1971information} and Ho \cite{ho1972team,ho1980team}. Witsenhausen's famous counterexample \cite{witsenhausen1968counterexample} vividly demonstrated the computational difficulties associated with team decision making in dynamic and stochastic environments. This still-unsolved counterexample described a simple team decision problem in which a nonlinear strategy strictly outperforms the optimal linear strategy and established deep connections between control, communication, and information theory.
Research on decentralized and distributed control theory has continued, and there has been a recent resurgence of interest driven by the advent of large-scale cyber-physical networks. Recent work has elaborated on connections with communication and information theory \cite{grover2013approximately,grover2015information} and focused on computational and structural issues \cite{rotkowitz2006characterization,nayyar2013decentralized,yuksel2013stochastic,swigart2014optimal,lessard2015optimal,gattami2012robust}.

These important structural information aspects arising from cooperating agents have received much less attention in the dynamic game literature, which tends to focus only on non-cooperative and adversarial behavior. The most well studied case is the  two-player problem, which features two opposing agents who have centralized information structures and has connections to robust control \cite{bacsar2008h}. What is currently under explored  is a comprehensive study of information structure aspects when there are \emph{both} non-cooperative elements, as in general dynamic game theory, and cooperative elements, as in decentralized control theory. These aspects can be captured by a two-team stochastic dynamic game framework.

Two-team stochastic dynamic games feature two opposing teams with decentralized information structures for both the attacking and defending teams:  each agent must act based on partial information measured or received locally in a way that coordinates its actions with team members and counters against the opposing team. This framework mathematically formalizes ``red team-blue team'' scenarios that qualitatively assess network security and resilience. In comparison to decentralized control theory, a team adversarial element is added. In comparison to general dynamic game theory, a sharp contrast between cooperation with teammates and conflict against adversaries is preserved. Further, the stochastic element (modeled by a ``chance'' or ``Nature'' player in the game) allows the inclusion of random component failures and disturbance signals.

There is currently a lack of deep theoretical and computational understanding in this class of dynamic games. A static version of the problem was studied in \cite{colombino2015quadratic}. Many fundamental questions that been answered in static or single team decentralized control settings do not have counterparts in the two-team setting.

The main contributions of the present paper are as follows. We formulate a two-team stochastic dynamic game problem and introduce a concept of mutual quadratic invariance, which defines a class of interacting information structures for the two teams under which optimal linear feedback control strategies are easy to compute. This is analogous to the concept of quadratic invariance in (single team) decentralized control \cite{rotkowitz2006characterization}. We show that for zero-sum two-team dynamic games, structured state feedback saddle point equilibrium strategies can be computed from equivalent structured disturbance feedforward saddle point equilibrium strategies. However, for nonzero-sum games we show via a counterexample that a similar equivalence fails to hold for structured Nash equilibrium strategies. Finally, we present a numerical example, which illustrates the importance of the information structure on the value of the game.

The rest of the paper is structured as follows. Section II provides preliminaries on static team games. Section III formulates a two-team stochastic dynamic game. Sections IV and V develop results on disturbance feedforward and state feedback strategies and introduce the concept of mutual quadratic invariance. Section VI presents illustrative numerical experiments, and Section VII gives some concluding remarks and future research directions. 

\section{Two-team stochastic static games}

In this section we review basic results for a two-team stochastic static game. In this setting, two teams who both know the distribution parameters of a Gaussian random vector $w$ need to decide strategies to compute vectors $u\in\R^m$ and $v\in\R^q$ as a function of the realization $w\in\R^n$ in order to minimize the expectation of \textit{different} quadratic forms in $w,u,v$. Each team is composed of multiple agents, each of which observes a different linear function of $w$ and decides a portion of the vectors $u$ or $v$. 

More formally, given a vector $w\sim\mc{N}(m_w,\Sigma_w)$, where $\Sigma_w\succ0$, consider the following game
\begin{equation}\label{eqn:team:game}
T_1: \left\{\begin{split}
\min_{\kappa_i(\cdot)} &\,\,\mathbb{E}_w\left(J_1(w,u,v)\right)\\
\text{s. t. }  &\; u_i=\kappa_i(C_iw)\\
& \forall i\in\Z_{[1,N]}
\end{split}\right., \;\;
T_2: \left\{\begin{split}
\min_{\lambda_j(\cdot)} &\,\,\mathbb{E}_w\left(J_2(w,u,v)\right)\\
\text{s. t. } &\; v_j=\lambda_j(\Gamma_iw) \\  
& \forall j\in\Z_{[1,M]}
\end{split}\right.
\end{equation}
where $J_i(w,u,v):=$
\begin{multline*}
\left[
\begin{array}{c}
 w \\
 u \\
 v  
\end{array}
\right]^\top
\left[
\begin{array}{ccc}
 \mc H_{i\;ww} &   \mc H_{i\;wu} &   \mc H_{i\;wv}  \\
   \mc H_{i\;wu}^\top &   \mc H_{i\;uu}  &   \mc H_{i\;uv}  \\
   \mc H_{i\;wv}^\top &   \mc H_{i\;uv}^\top  &    \mc H_{i\;vv} 
\end{array}
\right]
\left[
\begin{array}{c}
 w \\
 u \\
 v  
\end{array}
\right],\\ i\in\{1,2\}
\end{multline*}
are the objective functions of each team, and $\kappa_i(\cdot), i\in\Z_{[1,N]}$ and  $ \lambda_j(\cdot), j\in\Z_{[1,M]}$ are Borel measurable functions corresponding to the decision strategies of agents on team 1 and 2, respectively.

\begin{assumption}\label{ass:saddle}
We assume
$$
\left[
\begin{array}{cc}
 \mc H_{1\,uu} &  \mc H_{1\,uv}\\
 \mc H_{1\,uv}^\top&  \mc H_{2\,vv}
\end{array}
\right]\succ 0, \left[
\begin{array}{cc}
 \mc H_{1\,uu} &  \mc H_{2\,uv}\\
 \mc H_{2\,uv}^\top &  \mc H_{2\,vv}
\end{array}
\right]\succ 0.
$$
\end{assumption}
Note that in the zero-sum case $(J_1=-J_2)$ , Assumption~\ref{ass:saddle} is standard to guarantee the existence of a saddle point equilibrium to the game without decentralized information structure~\cite[condition 6.3.9]{hassibi1999indefinite}. If $J_1=J_2$, Assumption~\ref{ass:saddle} reduces to the standard positive definite assumption of team theory \cite{radner1962team}.

\vsp

We now define the set of Nash optimal strategies for the game in~\eqref{eqn:team:game}.

\vsp

\begin{definition}\label{def:nash}
A pair of strategies $(\kappa^\star(\cdot),\lambda^\star(\cdot))$ of the form $[\kappa_1^{\star\top}(C_1\,\cdot),\dots,\kappa_N^{\star\top}(C_n\,\cdot)]^\top$ and $[\lambda_1^{\star\top}(\Gamma_1\,\cdot),\dots,\lambda_M^{\star\top}(\Gamma_n\,\cdot)]^\top$ is Nash optimal for the game in~\eqref{eqn:team:game} if
\begin{equation}\label{eqn:nash}\left\{
\begin{array}{l}
\ds \kappa^\star(\cdot)\in\arg \min_{\kappa(\cdot)}\mathbb E_w J_1(w,\lambda^\star(w),\kappa(w))\\
\ds \lambda^\star(\cdot)\in\arg \min_{\lambda(\cdot)} \mathbb E_w J_2(w,\lambda(w),\kappa^\star(w)).
\end{array}
\right.
\end{equation}
\end{definition}

Under Assumption 1, the game in~\eqref{eqn:team:game} admits a unique set of linear Nash optimal strategies, which can be computed by solving a set of linear equations derived from stationarity conditions \cite{colombino2015quadratic}. This turns out to be a special case of a general multi-player, multi-objective linear quadratic static game considered in \cite{basar1978decentralized}.

\section{Two-team stochastic dynamic games}
Problem~\eqref{eqn:team:game} is a static game. There is no concept of time and causality of the information pattern. In this section we formulate a dynamic game, where two teams can influence the state evolution of a dynamical system. The agents on each team decide a portion of an input signal based on different observations of the system state over time. Decision must be causal: each player is only allowed to use past or, at most, present information. 

Our focus will be on the role of information structures for both teams in determining equilibrium strategies. Dynamic games offer a rich variety of information structures. Specific instances have been considered in the literature, with much work on various types of centralized structures \cite{basar1995dynamic} and some work on structures defined by spatiotemporal decentralization patterns. For example, a one-step-delay observation sharing pattern was shown in \cite{basar1978decentralized} to admit unique linear optimal strategies. There has been recent progress in (single team) decentralized control on information structure issues, including a characterization of information structures called quadratically invariant that yield convex control design problems \cite{rotkowitz2006characterization}. Here we seek an analogous result in a two-team game setting. 

Consider the system
\begin{equation}\label{eqn:dynamics}
\begin{split}
x(t+1) &=Ax(t)+B_1u(t)+B_2v(t)+w(t),
\end{split}
\end{equation}
where $x(t) \in\R^n$ is the system state at time $t$ with $x(0)\sim \mathcal{N}(0,\Sigma_0)$, $u(t) \in \R^{m_1}$ is the input for team 1 at time $t$, $v(t) \in \R^{m_2}$ is the input for team 2 at time $t$, and $w(t)\sim \mathcal{N}(0,\Sigma_t)$ is a random disturbance. The cost functions for each team are given by
\begin{multline}\label{eqn:cost:function}
J_i := \mathbb E \left(\sum_{t=0}^{N-1}x(t)^\top M_i(t) x(t) +  u(t)^\top R_i(t) u(t) \right .\\+ v(t)^\top V_i(t) v(t) \Bigg) + x(N)^\top M_i(N) x(N),\quad
 i\in\{1,2\},
\end{multline}
where $M_i(0)=\bold 0_{n\times n}$, $M_i(t)$ $=$ $M_i(t)^\top\in\R^{n\times n}$, $R_i(t)$ $=$ $R_i(t)^\top \in\R^{m_1\times m_1}$ and $V_i(t)$ $=$ $V_i(t)^\top \in\R^{m_2\times m_2}$.
By defining the matrices $\mc A$ $=$ $\blockdiag(A,\dots,A)$ $\in$ $\R^{n(N+1)\times n(N+1)}$,
$$
B_i = \left[
\begin{array}{ccc}
 B_i & 0  &  0 \\
 0 & \ddots  & 0  \\
 \vdots &   &   B_i\\
  0 & \cdots & 0
\end{array}
\right]\in\R^{n(N+1)\times m_jN}, \quad i\in\{1,2\},
$$
$\mc M_i$ $=$ $\blockdiag(0,M_i(1),\dots,M_i(N))$  $\in$ $\R^{n(N+1)\times n(N+1)}$, $\mc R_i$ $=$ $\blockdiag(R_i(0),\dots,R_i(N-1))$  $\in$ $\R^{m_1N\times m_1N}$ and $\mc V_i$ $=$ $\blockdiag(V_i(0),\dots,V_i(N-1))$  $\in$ $\R^{m_2N\times M_2N}$ for $i$~$\in$~$\{1,2\}$, the vectors $\bold x = (x(0),...,x(N))\in\R^{n(N+1)}$, $\bold u = (u(0),...,u(N-1))\in\R^{m_1N}$, $\bold v = (v(0),...,v(N-1))\in\R^{m_2N}$ and $\bold w = (x(0),w(0),...,w(N-1))\in\R^{n(N+1)}$, and the shift matrix 
$$
\mc Z:= \left[
\begin{array}{cccc}
0 &   & &  \\
 I & \ddots  & &  \\
  & \ddots  &  \ddots & \\
  &  & I  &  0 
  \end{array}\right]\in\R^{n(N+1)\times n(N+1)},
$$
we can write system~\eqref{eqn:dynamics} as
\begin{equation}\label{eqn:team:game:rewritten}
\begin{split}
\bold x &=\mc Z \mc A\bold x+ \mc Z  \mc B_1\bold u + \mc Z \mc B_2 \bold v+  \bold w.
\end{split}
\end{equation}
The system in~\eqref{eqn:team:game:rewritten} can be rewritten compactly as
\begin{equation}\label{eqn:team:game:rerewritten}
\bold{x}  
=
\left[
\begin{array}{ccc}
 \mc P_{11} &  \mc P_{12} &  \mc P_{13}
\end{array}
\right]
\left[
\begin{array}{c}
\bold{w}  \\
\bold{u}  \\
\bold{v}  
\end{array}
\right],
\end{equation}
where $\mc P_{11}=(I-\mc Z\mc A)^{-1} $, $\mc P_{12}=(I-\mc Z\mc A)^{-1} \mc  Z  \mc B_1$ and $\mc P_{13}=(I-\mc Z\mc A)^{-1} \mc  Z \mc  B_2$. The cost functions in~\eqref{eqn:cost:function} can be written in function of the vectorized inputs as
\begin{equation*}
\begin{split}
&\bold J_i(\bold u,\bold v)= \mathbb E_{\mathbf w}\left( 
\left[
\begin{array}{c}
\mathbf w \\
\mathbf u \\
\mathbf v  
\end{array}
\right]^\top
\mc H_i
\left[
\begin{array}{c}
\mathbf w \\
\mathbf u \\
\mathbf v  
\end{array}
\right]
\right),\quad i\in\{1,2\},
\end{split}
\end{equation*}
where
\begin{multline*}
\mc H_i =\\ \left[
\begin{array}{ccc}
\mc P_{11}^\top\mc M_i \mc  P_{11}& \mc P_{11}^\top\mc M_i\mc  P_{12} & \mc P_{11}^\top\mc M_i\mc  P_{13}  \\
 \mc P_{12}^\top\mc M_i \mc P_{11} &\mc P_{12}^\top\mc M_i\mc  P_{12} + \mc R_i &P_{12}^\top\mc M_i\mc  P_{13}  \\
 \mc P_{13}^\top\mc M_i \mc P_{11} &\mc  P_{13}^\top\mc M_i\mc  P_{12}  &\mc P_{13}^\top\mc M_i\mc  P_{13} + \mc V_i
\end{array}
\right],
\end{multline*}
for $i =\{1,2\}$.

We are interested in the finite horizon, two-team stochastic dynamic game where:
\begin {itemize}
\item Team 1 minimizes $\bold J_1(\bold u,\bold v)$ 
\item Team 2 minimizes $\bold J_2(\bold u,\bold v)$ 
\item Each team choses a structured causal state feedback strategy of the form
$$
\bold u = \mc K_1 (\bold x), \quad \bold v = \mc K_2 (\bold x). \quad \mc K_1\in\mc S_1,\;\mc  K_2\in \mc S_2,
$$
where $K_i:\R^{n(N+1)}\to\R^{m_iN}$, for $i\in\{1,2\}$ are measurable functions and $\mc S_1$ and $\mc S_2$ define an information structure for each team.
\end{itemize}

We define a information structure $\mc S_i\in\{0,1\}^{n(N+1)\times m_iN}$ as a binary matrix. $\mc K_i\in\mc S_i$ indicates that, if $[\mc S_i]_{jk}=0$, then the $j^\text{th}$ element of $\mc K_i$ is not a function of $\bold x_k$. By choosing the information structures one can enforce causality and a prescribed spatiotemporal structure on the controller strategies.

\section{Mutual Quadratic Invariance} \label{sec.dis.feedfarward}
In decentralized control with quadratically invariant information structures, the controller structure can be enforced on an affine parameter that defines the achievable set of closed-loop systems and recover a structured feedback controller. We now follow a similar approach in the two-team setting.

\subsection{Disturbance feedforward strategies}
By searching for measurable disturbance feedforward strategies of the type $\bold u=\mc Q_1(\mc P_{11}  \bold w)$ and $\bold v=\mc Q_2\mc (P_{11} \bold w)$, where $\mc Q_1\in \mc S_1$ and $\mc Q_2\in \mc S_2$, we recover the formulation of~\eqref{eqn:team:game}. Provided that Assumption~\ref{ass:saddle} is satisfied, there exists a unique Nash equilibrium of the form $\bold u=\bar {\mc Q}_1\mc P_{11} \bold w,~ \bold v= \bar{ \mc Q}_2\mc P_{11}  \bold w$ in the space of linear strategies \cite{colombino2015quadratic}. The matrices $\bar {\mc Q}_1, \bar{ \mc Q}_2$ can be easily computed by solving a linear system of equations or a sequence of semidefinite programs~\cite{colombino2015quadratic,basar1978decentralized}.
Assumption~\ref{ass:saddle} for the dynamic game problem becomes
\begin{align*}
\left[
\begin{array}{cc}
\mc P_{12}^\top\mc M_1\mc  P_{12} +\mc R_1 &\mc P_{12}^\top\mc M_1\mc  P_{13}\\
\mc P_{13}^\top\mc M_1\mc  P_{12} &  \mc P_{13}^\top\mc M_2\mc  P_{13} + \mc S_2
\end{array}
\right]\succ 0, \\
\left[
\begin{array}{cc}
\mc P_{12}^\top\mc M_2\mc  P_{12} +\mc R_2 &\mc P_{12}^\top\mc M_2\mc  P_{13}\\
\mc P_{13}^\top\mc M_2\mc  P_{12} &  \mc P_{13}^\top\mc M_1\mc  P_{13} + \mc S_1
\end{array}
\right]\succ 0.
\end{align*}
We can define new cost functions that depend on the matrices $\mc Q_1$ and $\mc Q_2$ describing linear disturbance feedforward strategies  as
\be\label{eq.cost.q}
\mc J_i\left(
\left[
\begin{array}{c}
{\mc Q}_1\\
{\mc Q}_2
\end{array}
\right]
 \right) = \bold J_i(\bold u,\bold v)\bigg|_{\bold u=\mc Q_1\mc P_{11} \bold w, \bold v=\mc Q_2\mc P_{11} \bold w}.
\ee
In particular,
\begin{multline*}
\mc J_i\left(
\left[
\begin{array}{c}
{\mc Q}_1\\
{\mc Q}_2
\end{array}
\right]
\right) = \|\mc M_i^{1\over 2} \left( I+\mc P_{12}\mc Q_1+\mc P_{13}\mc Q_2\right)\mc P_{11}\Sigma_{\bold w}^{1\over 2} \|^2_F \\
+ \| \mc R_i^{1\over 2} \mc Q_1\Sigma_{\bold w}^{1\over 2} \|^2_F 
 +  \|\mc V_i^{1\over 2} \mc Q_2 \Sigma_{\bold w}^{1\over 2}\|^2_F,
\end{multline*}
where $\Sigma_{\bold w}^{1\over 2}$ is the covariance of $\bold w$.
\subsection{Equivalent state feedback strategies}
It is easy to show that there exists a bijective relationship between a pair of linear disturbance feedforward strategies ($\mc Q_1, \mc Q_2$) and an equivalent pair of linear state feedback strategies described by the matrices ($\mc K_1, \mc K_2$). More precisely, using~\eqref{eqn:team:game:rerewritten} we obtain
\begin{equation}\label{eq.equivalent.k}
\begin{split}
\left[
\begin{array}{c}
\bold {u}\\
\bold {v}
\end{array}
\right]& =\left[
\begin{array}{c}
\mc Q_1\\
\mc Q_2
\end{array}
\right]\mc P_{11}\bold w \\
& =  \left[
\begin{array}{c}
\mc Q_1\\
\mc Q_2
\end{array}
\right]\bold x - \left[
\begin{array}{c}
\mc Q_1\\
\mc Q_2
\end{array}
\right]\left[
\begin{array}{cc}
\mc P_{12} & \mc P_{13}
\end{array}
\right]\left[
\begin{array}{c}
\bold u\\
\bold v
\end{array}
\right].
\end{split}
\end{equation}
We can define the function $g$ such that
\begin{align*}
 \left(
\left[
\begin{array}{c}
{\mc K}_1\\
{\mc K}_2
\end{array}
\right]
\right)
  = g
  \left(
 \left[
\begin{array}{c}
{\mc Q}_1\\
{\mc Q}_2
\end{array}
\right]
  \right),
 \end{align*}
where
\begin{multline*}
 g\left(
\left[
\begin{array}{c}
{\mc Q}_1\\
{\mc Q}_2
\end{array}
\right]
  \right) = \left( I + \left[
\begin{array}{c}
\mc Q_1\\
\mc Q_2
\end{array}
\right] \left[
\begin{array}{cc}
\mc P_{12} & \mc P_{13}
\end{array}
\right]\right)^{-1}
 \left[
\begin{array}{c}
\mc Q_1\\
\mc Q_2
\end{array}
\right].
\end{multline*}

Using a similar approach to~\eqref{eq.equivalent.k}, one can construct the inverse mapping that, given a pair of feedback strategies, recovers the equivalent feedforward strategies.

$$
 \left[
\begin{array}{c}
{\mc Q}_1\\
{\mc Q}_2
\end{array}
\right]
 = g^{-1}\left(
\left[
\begin{array}{c}
{\mc K}_1\\
{\mc K}_2
\end{array}
\right]
\right),
$$
where the map $g^{-1}$ takes the form
\begin{multline*}
  g^{-1}\left(
\left[
\begin{array}{c}
{\mc K}_1\\
{\mc K}_2
\end{array}
\right]
  \right)=\\
\left[
\begin{array}{c}
{\mc K}_1\\
{\mc K}_2
\end{array}
\right]\left( I - \left[
\begin{array}{cc}
\mc P_{12} & \mc P_{13}
\end{array}
\right]
\left[
\begin{array}{c}
{\mc K}_1\\
{\mc K}_2
\end{array}
\right]\right)^{-1}.
\end{multline*}

\vsp

Given a pair of linear feedback strategies $\mc K_1$ and $\mc K_2$, the cost for player $i$ can be evaluated by considering the equivalent feedforward strategies as
$$
 \bold J_i(\bold u,\bold v)\bigg|_{\bold u=\mc K_1\bold x, \bold v=\mc K_2\bold x} = \mc J_i \left(g^{-1}\left(
\left[
\begin{array}{c}
{\mc K}_1\\
{\mc K}_2
\end{array}
\right]
 \right)\right),
$$
where $\mc J_i$ is defined in~\eqref{eq.cost.q}. 

\vsp

Now that we have a way to construct feedback strategies which are equivalent to any set of linear feedforward strategies, we need to establish a condition that guarantees that such equivalent feedback strategies will preserve the desired structure.

\subsection{Mutual Quadratic Invariance}

\vsp

We know form the quadratic invariance literature~\cite{rotkowitz2006characterization,swigart2010explicit} that $\mc Q_1\in \mc S_1$ and $\mc Q_2\in \mc S_2$ $\iff$ $\mc K_1\in \mc S_1$ and $\mc K_2\in \mc S_2$ if and only if  for all $
\left(\mc K_1,
\mc K_2\right )
\in \mc S_1 \times \mc S_2$
it holds that
\begin{equation}\label{eqn:mut:quad:inv:1}
\left[
\begin{array}{c}
\mc K_1\\
\mc K_2
\end{array}
\right]
\left[
\begin{array}{cc}
\mc P_{12} & \mc P_{13}
\end{array}
\right]\left[
\begin{array}{c}
\mc K_1\\
\mc K_2
\end{array}
\right]\in \mc S_1 \times \mc S_2,
\end{equation}
in other words $\mc S_1 \times \mc S_2$ is quadratically invariant under $\left[ \;\mc P_{12} \quad \mc P_{13} \; \right ]$. We define this property as \textbf{mutual quadratic invariance}. We can expand~\eqref{eqn:mut:quad:inv:1} as
\begin{equation}\label{eqn:mut:quad:inv:2}
\left[
\begin{array}{c}
\mc K_1\\
\mc K_2
\end{array}
\right]\in \mc S_1 \times \mc S_2\implies
\left\{
\begin{array}{c}
\mc K_1\mc P_{12} \mc K_1 , \quad  \mc K_1\mc P_{13} \mc K_2 \in\mc S_1\\
\mc K_2\mc P_{12} \mc K_1 , \quad  \mc K_2\mc P_{13} \mc K_2 \in\mc S_2.
\end{array}
\right.
\end{equation}

By observing~\eqref{eqn:mut:quad:inv:1} we note that MQI is equivalent to QI for a control problem where both decisions $\bold u$ and $\bold v$ are taken by a single decision maker. MQI information structures will allow us to compute equilibrium strategies in two-team games.

\section{Computing Equilibrium Strategies}
We have seen in Section~\ref{sec.dis.feedfarward} that given structures $\mc S_1$ and $\mc S_2$ which are mutually quadratically invariant under $\mc P_{12}$ and $\mc P_{13}$, one can easily obtain a Nash equilibrium in the disturbance feedforward strategies. Furthermore once we recover the equivalent state feedback strategies the structure is preserved. There is still a nontrivial question that needs to be answered.

\vsp

\begin{problem}\label{prob.nqnk}
Given a Nash equilibrium in the feedforward strategies ($\bar {\mc Q}_1, \bar{ \mc Q}_2$), are the equivalent feedback strategies 
\begin{align}\label{eqn:k_bar}
 \left(
\left[
\begin{array}{c}
\bar{\mc K}_1\\
\bar{\mc K}_2
\end{array}
\right]
\right)
  = g
  \left(
\left[
\begin{array}{c}
\bar{\mc Q}_1\\
\bar{\mc Q}_2
\end{array}
\right]
  \right),
 \end{align}
 a Nash equilibrium in the feedback strategies? 
\end{problem}

\vsp 

In order to understand why the answer to Problem~\ref{prob.nqnk}  is nontrivial, we first present a counterexample for a nonzero sum game.

\subsection{Nonzero sum game}
Consider the following problem instance with
\begin{equation}
\begin{aligned}
&N = 2, \quad A = 2, \quad B_1 = 0.4, \quad B_2 = 0.1, \\
&\Sigma_0 = 1, \quad \Sigma_t = 1 \ \forall t, \\
&M_1(t) = R_1(t) = S_1(t) = 1 \ \forall t, \\
&M_2(t) = 70 \ \forall t, R_2(t) = S_2(t) = 1 \ \forall t.
\end{aligned}
\end{equation}
This is a single state, two-stage, two-player problem, and we will consider centralized, causal strategies, which are readily verified to be mutually quadratically invariant. Using disturbance feedforward strategies, the problem can be reduced to a static game whose unique Nash Equilibrium strategies are readily computed using methods from~\cite{basar1978decentralized,colombino2015quadratic}. This yields the Nash pair $(\bar{\mc Q}_1,\bar{ \mc Q}_2)$, where

\begin{equation}
\begin{aligned}
\bar {\mc Q}_1 &= \left[\begin{array}{ccc}-0.6795 & 0 & 0 \\0.6283 & -0.4301 & 0\end{array}\right], \\
\bar{\mc Q}_2 &= \left[\begin{array}{ccc}-11.890 & 0 & 0 \\10.996 & -7.5269 & 0\end{array}\right].
\end{aligned}
\end{equation}
The corresponding equilibrium value for player 1 is $\mc J_1^*\left( \left[
\begin{array}{c}
\bar{\mc Q}_1\\
\bar{\mc Q}_2
\end{array}
\right] \right) = 220$. The corresponding state feedback strategies are
\begin{equation}
\begin{aligned}
\bar{\mc K}_1 &= \left[\begin{array}{ccc}-0.6795 & 0 & 0 \\ 0 & -0.4301 & 0\end{array}\right], \\
\bar{\mc K}_2 &= \left[\begin{array}{ccc}-11.890 & 0 & 0 \\ 0 & -7.5269 & 0\end{array}\right],
\end{aligned}
\end{equation}
However, $(\bar{\mc K}_1 ,\bar{\mc K}_2 )$ is not a Nash Equilibrium in state feedback strategies since it is readily verified that $\mc J_1\left(g^{-1}\left(
\left[
\begin{array}{c}
{ \hat{\mc K}}_1\\
\bar{ \mc K}_2
\end{array}
\right]
 \right)\right) = 206.1$, where
\begin{equation}
\hat{\mc K}_1 = \left[\begin{array}{ccc}-1.853 & 0 & 0 \\ 0 & -0.4301 & 0\end{array}\right].
\end{equation}
Although the sparsity structure is preserved between the disturbance feedforward and state feedback strategies since the information structure is mutually quadratically invariant, the Nash Equilibrium property is not preserved. Thus, for this particular non-zero sum dynamic game, the answer to the question posed in Problem 1 is negative. We show next that this situation does not occur in zero-sum games.

\subsection{Zero sum game}
Let us now consider the zero sum game case, where the objective function of one team is precisely the negative of that of the other team. Such problems can be re-written as a min-max problem of the form.

\begin{equation}\label{eq.zsg}
\begin{split}
\min_{\bold u}\max_{\bold w} \quad& \bold J(\bold u,\bold v) := \mathbb E_{\mathbf w}\left( 
\left[
\begin{array}{c}
\mathbf w \\
\mathbf u \\
\mathbf v  
\end{array}
\right]^\top
\mc H
\left[
\begin{array}{c}
\mathbf w \\
\mathbf u \\
\mathbf v  
\end{array}
\right]
\right)\end{split}
\end{equation}

As before we are interested in strategies of the form
$$
\bold u = \mc K_1 (\bold x), \quad v = \mc K_2 (\bold x). \quad \mc K_1\in\mc S_1,\;\mc  K_2\in \mc S_2,
$$
where $\mc S_1$ and $\mc S_2$ are prescribed sets of structured causal controllers. As zero-sum games are a special case of nonzero sum games, provided Assumption~\ref{ass:saddle} holds, we can find structured Nash equilibria (rather saddle point equilibrium in the zero sum context) if we consider linear strategies of the form $\bold u=\mc Q_1\mc P_{11} \bold w$ and $\bold v=\mc Q_2\mc P_{11} \bold w$.
For the zero sum game of the form~\eqref{eq.zsg}, Assumption~\ref{ass:saddle} reads
$$
\left[
\begin{array}{cc}
\mc P_{12}^\top\mc M\mc  P_{12} +\mc R &\mc P_{12}^\top\mc M\mc  P_{13}\\
\mc P_{13}^\top\mc M\mc  P_{12} &  -\left(\mc P_{13}^\top\mc M\mc  P_{13} + \mc S\right)
\end{array}
\right]\succ 0. 
$$
As before, given the saddle point equilibrium in the feedforward strategies, when $\mc S_1$ and $\mc S_2$  are mutually quadratically invariant, we can compute equivalent linear feedback strategies that preserve the structure. In the zero-sum case, however, we can relate the equilibrium property of the feedforward strategies to that of the state feedback strategies.

\vsp

We start with a result that shows that the maps $g$ and $g^{-1}$ preserve stationary points. We begin by noting that for a zero-sum game $\mc J_1=-\mc J_2$. 
\begin{lemma}\label{lem.lemma_stationarity}
Given 
\begin{align*}
 \left[
\begin{array}{c}
\bar {\mc Q}_1\\
\bar {\mc Q}_2
\end{array}
\right]
 & = g^{-1}\left(
  \left[
\begin{array}{c}
\bar{\mc K}_1\\
\bar{\mc K}_2
\end{array}
\right]
  \right)
\end{align*}
then 
\be
\begin{split}\label{thm:thesis}
\frac{\partial \mc J_1\left (
\left[
\begin{array}{c}
{\mc Q}_1\\
{\mc Q}_2
\end{array}
\right]
\right )}{\partial \mc Q_1}\bigg|_{\mc Q_1 = \bar {\mc Q}_1 ,\mc Q_2 = \bar {\mc Q}_2} & \in \mc S_1^\perp, \\
\frac{\partial \mc J_2\left(
\left[
\begin{array}{c}
{\mc Q}_1\\
{\mc Q}_2
\end{array}
\right]
\right )}{\partial \mc Q_2}\bigg|_{\mc Q_1 = \bar {\mc Q}_1 ,\mc Q_2 = \bar {\mc Q}_2} & \in \mc S_2^\perp,
\end{split}
\ee
if and only if 
\be
\begin{split}\label{thm:thesis2}
\frac{\partial \mc J_1\left (
g^{-1}\left(\left[
\begin{array}{c}
{\mc K}_1\\
{\mc K}_2
\end{array}
\right]
\right)
\right )}{\partial \mc K_1}\bigg|_{\mc K_1 = \bar {\mc K}_1 ,\mc K_2 = \bar {\mc K}_2} & \in \mc S_1^\perp, \\
\frac{\partial \mc J_2\left(
g^{-1}\left(\left[
\begin{array}{c}
{\mc K}_1\\
{\mc K}_2
\end{array}
\right]
\right)
\right )}{\partial \mc K_2}\bigg|_{\mc K_1 = \bar {\mc K}_1 ,\mc K_2 = \bar {\mc K}_2} & \in \mc S_2^\perp,
\end{split}
\ee
\end{lemma}
\begin{proof}
Let us simplify the notation and define 
$$
  {\mc Q} :=  \left[
\begin{array}{c}
 {\mc Q}_1\\
 {\mc Q}_2
\end{array}
\right],\quad  {\mc K} :=  \left[
\begin{array}{c}
 {\mc K}_1\\
 {\mc K}_2
\end{array}
\right], \quad 
  {\mc P} :=  \left[
\begin{array}{cc}
\mc P_{12} & \mc P_{13}
\end{array}
\right].
$$
We start by proving the \emph{only if} part. 
Note that since $\mc J_1 = -\mc J_2$,
$$
\frac{\partial \mc J_2\left(
\left[
\begin{array}{c}
{\mc Q}_1\\
{\mc Q}_2
\end{array}
\right]
\right )}{\partial \mc Q_2}\bigg|_{\mc Q_1 = \bar {\mc Q}_1 ,\mc Q_2 = \bar {\mc Q}_2} \in \mc S_2^\perp
$$
if and only if
$$
\frac{\partial \mc J_1\left(
\left[
\begin{array}{c}
{\mc Q}_1\\
{\mc Q}_2
\end{array}
\right]
\right )}{\partial \mc Q_2}\bigg|_{\mc Q_1 = \bar {\mc Q}_1 ,\mc Q_2 = \bar {\mc Q}_2}  \in \mc S_2^\perp.
$$
Assume~\eqref{thm:thesis2} holds, and suppose~\eqref{thm:thesis} does not hold. Then there exist $\tilde{\mc {Q}}\in \mc S_1 \times  \mc S_2$, with $\tilde{\mc {Q}}\neq0$ such that
$$
\lim_{\varepsilon \to 0}\frac{\mc J_1(\bar{\mc Q}+\varepsilon \tilde{\mc Q})-\mc J(\bar{\mc Q})}{\varepsilon} = \kappa \neq 0,
$$
or equivalently
\be \label{eqn:g_ginv}
\lim_{\varepsilon \to 0}\frac{\mc J_1\left(g^{-1}\left(g\left(\bar{\mc Q}+\varepsilon \tilde{\mc Q}\right)\right )\right)-\mc J\left(\bar{\mc Q}\right)}{\varepsilon} = \kappa \neq 0.
\ee
We know that 
\be
\begin{split}\label{eq.variation.g}
& g\left(\bar{\mc Q}+\varepsilon \tilde{\mc Q}\right) = \left (I +\left (\bar{\mc Q}+\varepsilon \tilde{\mc Q}\right)\mc P\right  )^{-1}\left (\bar{\mc Q}+\varepsilon \tilde{\mc Q}\right) \\ 
& = \left (I +\bar{\mc Q}\mc P+\varepsilon \tilde{\mc Q}\mc P\right )^{-1}\left (\bar{\mc Q}+\varepsilon \tilde{\mc Q}\right)\\
& = \left [ (I +\bar{\mc Q}\mc P)^{-1} + \varepsilon  (I +\bar{\mc Q}\mc P)^{-1} \tilde{\mc Q} \mc P (I +\bar{\mc Q}\mc P)^{-1}\right.\\
& + \mc O(\varepsilon ^2) \Big]\left (\bar{\mc Q}+\varepsilon \tilde{\mc Q}\right)\\
& = \bar{\mc K} + \varepsilon \tilde {\mc K} + \mc O(\varepsilon ^2),
\end{split}
\ee
where
$$
\tilde {\mc K} = (I +\bar{\mc Q}\mc P)^{-1}\tilde{\mc Q} + (I +\bar{\mc Q}\mc P)^{-1} \tilde{\mc Q} \mc P (I +\bar{\mc Q}\mc P)^{-1} \bar{\mc Q}.
$$

Using~\cite[Theorem 14 + Theorem 26]{rotkowitz2006characterization}, and mutual quadratic invariance it is easy to conclude that,  $\tilde {\mc K}\in\mc S_1 \times \mc S_2$. Substituting~\eqref{eq.variation.g} in~\eqref{eqn:g_ginv} and using the fact that $\bar{\mc{Q}} = g\left (\bar{\mc K}\right )$ one obtains
\begin{equation*} \label{eqn:g_g_K}
\begin{split}
&\lim_{\varepsilon \to 0}\frac{\mc J_1\left(g^{-1}\left( \bar{\mc K} + \varepsilon \tilde {\mc K} + \mc O(\varepsilon ^2)\right )\right)-\mc J\left( g^{-1}\left (\bar{\mc K}\right )\right)}{\varepsilon} = \\ 
&\lim_{\varepsilon \to 0}\frac{\mc J_1\left(g^{-1}\left( \bar{\mc K} + \varepsilon \tilde {\mc K} \right )\right)-\mc J\left( g^{-1}\left (\bar{\mc K}\right )\right)}{\varepsilon}
 = \kappa \neq 0,
\end{split}
\end{equation*}
which, since $\tilde {\mc K}\in\mc S_1 \times \mc S_2$, is in contradiction with~\eqref{thm:thesis2} and thus proves the claim.

The converse direction can be proven analogously. 
\end{proof}

\vsp

Note that Lemma~\ref{lem.lemma_stationarity} only holds for zero-sum games as the proof heavily relies to the fact that $\mc J_1 = -\mc J_2$. We are now ready to state our main result, which allows us to construct a saddle point Equilibrium in the space of linear feedback strategies.

\vsp

\begin{theorem}\label{thm.nash.equivalent}
Let ($\bar {\mc Q}_1,\bar {\mc Q}_2$)  be the unique saddle point equilibrium in the disturbance feedforward strategies. Then if there exists a saddle point equilibrium in state feedback strategies, it is unique and given by
\begin{align*}
 \left[
\begin{array}{c}
\bar {\mc K}_1\\
\bar {\mc K}_2
\end{array}
\right]
 & = g\left(
  \left[
\begin{array}{c}
\bar{\mc Q}_1\\
\bar{\mc Q}_2
\end{array}
\right]
  \right).
\end{align*}
\end{theorem}

\begin{proof}

Let $(\hat {\mc K}_1,\hat {\mc K}_2)$ be any saddle point equilibrium in state feedback linear strategies. Let $(\hat {\mc Q}_1,\hat {\mc Q}_2)$ be the corresponding disturbance feedforward policy. Since $(\hat {\mc K}_1,\hat {\mc K}_2)$ is stationary, so is $(\hat {\mc Q}_1,\hat {\mc Q}_2)$ by Lemma 1. Since $\mathbf J$ is convex quadratic in $ {\mc Q}_1$ and concave quadratic in $ {\mc Q}_2$, it follows from Assumption 1 that the stationary point is unique. Thus, $(\hat {\mc Q}_1,\hat {\mc Q}_2) = (\bar {\mc Q}_1,\bar {\mc Q}_2)$. Since $g$ is bijective, $(\hat {\mc K}_1,\hat {\mc K}_2) = (\bar {\mc K}_1,\bar {\mc K}_2)$. Thus, $(\bar {\mc K}_1,\bar {\mc K}_2)$ is the unique saddle point equilibrium in state feedback linear strategies due to the uniqueness of ($\bar {\mc Q}_1,\bar {\mc Q}_2$).
\end{proof}

This result allows computation of structured equilibrium feedback strategies in two-team stochastic dynamic games with mutually quadratically invariant information structures. 

\section{Numerical example}

We now present a very simple illustrative numerical example. The mutual quadratic invariance property allows us to compute equilibrium strategies and values under these information structures and thereby to understand the importance of different information structures in dynamical games. We consider a two player system depicted in Figure~\ref{fig:int}, which can be interpreted as a simple transportation network. 

\begin{figure}[htp!]
\begin{center}
\begin{tikzpicture}[>=stealth',shorten >=1pt, node distance=1.8cm, on grid, initial/.style={}]

 \node[state,fill=RoyalBlue]          (1)                        {$\textcolor{white}{x_1}$};
 \node[state,fill=Maroon]          (2) [right =of 1]    {$\textcolor{white}{x_2}$};
 \node[] (3) [right =of 2]    {};
 \node[] (4) [above =of 1] {};
 \node[] (5) [above =of 2] {};

\path
 (1)     edge [->, color = ProcessBlue, double=ProcessBlue]  node [midway, above] {$u$}  (2)
 (2)     edge [->, color = red, double  = red]  node [midway, above] {$v$}  (3)
 (4)     edge [->, color = OliveGreen, double  = OliveGreen]  node [midway, left] {$w_1$}  (1)
 (5)     edge [->, color = OliveGreen, double  = OliveGreen]  node [midway, left] {$w_2$}  (2);

\end{tikzpicture} 
\caption{Two player network}
\label{fig:int}
\end{center}
\end{figure}

\vsp

Each subsystem consists of a buffer with single integrator dynamics. System 1 stores $x_1$ and can control input $u$ that transfers some of the good stored in its buffer to System 2. System 2 stores $x_2$ and can control input $v$ to discard some of the good. Both systems are affected by random disturbances which are normally distributed with zero mean and unit variance. The dynamics of the system is

\begin{multline}\label{eq.game.dyn}
\left[
\begin{array}{c}
{x_1^+} \\
{x_2^+} 
\end{array}
\right] = 
\left[
\begin{array}{cc}
1 & 0 \\
0 & 1
\end{array}
\right] 
\left[
\begin{array}{c}
{x_1} \\
{x_2} 
\end{array}
\right] + \\
\left[
\begin{array}{cc}
-1 \\
1
\end{array}
\right] {u} +
\left[
\begin{array}{cc}
0 \\
-1
\end{array}
\right] {v} 
+
\left[
\begin{array}{ccc}
{w_1} \\
{w_2 } 
\end{array}
\right] 
\end{multline}

Given the dynamics in~\eqref{eq.game.dyn} with $w_1(t),w_2(t), x(0)\sim\mc N(0,1)$ we consider the following zero sum dynamic game

\begin{equation}\label{eq.game.example}
\begin{array}{rl}
\ds \min_{{\bold u}} \max_{{\bold v}} & \ds\sum_{t=1}^{10} 2 \, {\mathbb E  x_1^2(t)} + {\mathbb E  u^2(t-1)}  + \\
&\quad \quad - \,  {\mathbb E  x_2^2(t)} -2\,{\mathbb E  v^2(t-1)} \\
\\
\ds \textnormal{s. t.} & \bold u = \mc K_1 \bold x,\quad \mc K_1\in\mc S_1 \\
\ds 				    & \bold v = \mc K_2 \bold x,\quad \mc K_2\in\mc S_2, 		    
\end{array}
\end{equation}
where $\bold x = (x(0),...,x(10))$, $\bold u = (u(0),...,u(9))$, $\bold v = (v(0),...,v(9))$. Both players benefit from keeping the variance of their state and input low and from increasing the variance of the opponent's state and input. We will compare the results for different information structures $\mc S_1$ and $\mc S_2$, all of which are mutually quadratically invariant. In particular, we consider:

\begin{itemize}
\item Causal controllers with full information (FI). That is both players have access to all past and present information.
$$
 {\mc S_1 = \mc S_2 =  \left[
\begin{array}{ccccccccccccccccc}
\star & \star & \\ 
\star & \star &\star & \star &\\
\star & \star &\star & \star &\star & \star &\\
\star & \star &\star & \star &\star & \star &\star & \star &\\
\star & \star &\star & \star &\star & \star &\star & \star &\star & \star & \\
&&&&&\vdots \\
\end{array}
\right]. }
$$

\item One step delay information sharing (1SDIS). At time $t$ both players do not know the opponent's current state but have full information up to time $t-1$. 
$$
 {\mc S_1 =   \left[
\begin{array}{ccccccccccccccccc}
\star & 0 & \\ 
\star & \star & \star & 0 &\\
\star & \star &\star & \star &\star & 0 &\\
\star & \star &\star & \star &\star & \star &\star & 0 &\\
\star & \star &\star & \star &\star & \star &\star & \star &\star & 0 &   \\
&&&&&\vdots \\
\end{array}
\right], }
$$
$$
 {\mc S_2 =   \left[
\begin{array}{ccccccccccccccccc}
0 & \star & \\ 
\star & \star & 0 & \star &\\
\star & \star &\star & \star &0 & \star &\\
\star & \star &\star & \star &\star & \star &0 & \star &\\
\star & \star &\star & \star &\star & \star &\star & \star &0 & \star &   \\
&&&&&\vdots \\
\end{array}
\right]. }
$$

\item Decentralized control for Player 1 (DP1). Player 1 only has access to present and past information about its own state, Player 2 has full information.

$$
 {\mc S_1 =   \left[
\begin{array}{ccccccccccccccccc}
\star & 0 & \\ 
\star & 0 & \star & 0 &\\
\star & 0 &\star & 0 &\star & 0 &\\
\star & 0 &\star & 0 &\star & 0 &\star & 0 &\\
\star & 0 &\star & 0 &\star & 0 &\star & 0 &\star & 0 &   \\
&&&&&\vdots \\
\end{array}
\right], }
$$

and $\mc S_2$ has full information.
\end{itemize}

It is easily verifiable that all such structures respect the mutual quadratic invariance assumption for the given system. We computed the Nash equilibrium feedforward strategies ($\bar {\mc Q}_1,\bar {\mc Q}_2$) for the three different information structures using the method proposed in~\cite{colombino2015quadratic}. We applied Theorem~\ref{thm.nash.equivalent} to compute the linear saddle point equilibrium in the state feedback strategies as
\begin{align*}
 \left[
\begin{array}{c}
\bar {\mc K}_1\\
\bar {\mc K}_2
\end{array}
\right]
 & = g\left(
  \left[
\begin{array}{c}
\bar{\mc Q}_1\\
\bar{\mc Q}_2
\end{array}
\right]
  \right).
\end{align*}
In Table~\ref{table:Table} we observe the different costs at equilibrium. Note the large difference in cost function that is achieved for different information structures. Using the full information structure as a baseline, as we expect, Player 1 is penalized by using decentralized information (DP1). On the other hand, Player 1 obtains a great advantage with the one step delay information sharing (1SDIS) structure. 
\begin{table}
\caption{The cost at the Nash Equilibrium for the three different information structures. A smaller cost is indicates an advantage for Player 1 who is minimizing in~\eqref{eq.game.example}, while a larger cost is an advantage for Player 2.}
\label{table:Table}
\begin{center}
\begin{tabular}{ @{} c|c@{} } 
\bottomrule
\multicolumn{2}{>{\columncolor[gray]{.95}}c}{Equilibrium cost for the different information structures} \\
\toprule

Structure & Equilibrium Cost~\eqref{eq.game.example}   \\
\toprule 
FI  &   -1.58    \\
1SDIS   & -10.02   \\
DP1     &   0.00 \\
\toprule
\bottomrule
\end{tabular}
\end{center}
\end{table}

Figure~\ref{fig:sym} shows the breakdown of the cost function of~\eqref{eq.game.example} for the different information structures and allows us to understand why certain information structures are more beneficial for different players. For example, if we consider 1SDIS we notice that Player 1 can exploit the fact that its opponent has no information on Player 1's current state and input and it uses this to dramatically increase the variance of $x_2$. To do so Player 1 needs to `spend' some variance in $u$, which is also increased. The net gain, however, is clearly in favor of Player 1. 

\begin{figure}[htp!]
\begin{center}
\begin{tikzpicture}
\begin{axis}[
    ybar,
    enlargelimits=0.15,
    legend style={at={(0.5,-0.15)},
      anchor=north,legend columns=-1},
    symbolic x coords={FI,1SDIS,DP1},
    xtick=data
    ]
\addplot[draw=none,mark=none,fill=RoyalBlue] coordinates {(FI,21.9596) (1SDIS,22.0122) (DP1,20.5012)};
\addplot[draw=none,mark=none,fill=ProcessBlue] coordinates {(FI,8.2812) (1SDIS,11.8580) (DP1,7.7873)};
\addplot[draw=none,mark=none,fill=Maroon] coordinates {(FI,18.3332) (1SDIS,30.3471) (DP1,16.4917)};
\addplot[draw=none,mark=none,fill=red] coordinates {(FI, 13.5872) (1SDIS,13.5168) (DP1, 11.8868)};

\legend{$2\sum\mathbb E x_1(t)^2$,$\sum\mathbb E u(t)^2$,$\sum\mathbb E x_2(t)^2$,$2\sum\mathbb E v(t)^2$}
\end{axis}
\end{tikzpicture}
\caption{Breakdown of the cost function of~\eqref{eq.game.example} for the three information structures}
\label{fig:sym}
\end{center}
\end{figure}

\section{Conclusion}
We have considered a two-team linear quadratic stochastic dynamic game with decentralized information structures for both teams. We introduced the concept of Mutual Quadratic Invariance (MQI), which defines a class of interacting team information structures for which equilibrium strategies can be easily computed. We demonstrated an equivalence of disturbance feedforward and state feedback saddle point equilibrium strategies that facilitates this computation in zero-sum games, and we showed such an equivalence fails to hold for Nash equilibrium strategies in nonzero-sum games. A numerical example showed how mutually quadratically invariant information structures can be evaluated and how different structures can lead to significantly different equilibrium values.

Many fundamental questions remain open in two-team stochastic dynamic games. For example, issues involving infinite horizon and boundedness of the equilibrium value (stability), separation and certainty equivalence, games with incomplete model information, and design of information structures can be considered. Some of these results may take inspiration from recent progress on information structure issues in decentralized control.

\bibliography{bib_file}

\begin{thebibliography}{10}

\bibitem{basar1995dynamic}
T.~Ba{\c{s}}ar and G.~Olsder, {\em Dynamic noncooperative game theory},
  vol.~200.
\newblock SIAM, 1995.

\bibitem{von1944theory}
J.~Von~Neumann and O.~Morgenstern, {\em Theory of games and economic behavior}.
\newblock Princeton University Press, 1944.

\bibitem{bellman1952theory}
R.~Bellman, ``On the theory of dynamic programming,'' {\em Proceedings of the
  National Academy of Sciences of the United States of America}, vol.~38,
  no.~8, p.~716, 1952.

\bibitem{pontryagin1959optimal}
L.~Pontryagin, ``Optimal control processes,'' {\em Usp. Mat. Nauk}, vol.~14,
  no.~3, 1959.

\bibitem{marschak1955elements}
J.~Marschak, ``Elements for a theory of teams,'' {\em Management Science},
  vol.~1, no.~2, pp.~127--137, 1955.

\bibitem{radner1962team}
R.~Radner, ``Team decision problems,'' {\em The Annals of Mathematical
  Statistics}, pp.~857--881, 1962.

\bibitem{marschak1972economic}
J.~Marshak and R.~Radner, {\em Economic theory of teams}.
\newblock Yale University Press, 1972.

\bibitem{witsenhausen1968counterexample}
H.~Witsenhausen, ``A counterexample in stochastic optimum control,'' {\em SIAM
  Journal on Control}, vol.~6, no.~1, pp.~131--147, 1968.

\bibitem{witsenhausen1971separation}
H.~Witsenhausen, ``Separation of estimation and control for discrete time
  systems,'' {\em Proceedings of the IEEE}, vol.~59, no.~11, pp.~1557--1566,
  1971.

\bibitem{witsenhausen1971information}
H.~Witsenhausen, ``On information structures, feedback and causality,'' {\em
  SIAM Journal on Control}, vol.~9, no.~2, pp.~149--160, 1971.

\bibitem{ho1972team}
Y.~C. Ho and K.-C. Chu, ``Team decision theory and information structures in
  optimal control problems--part i,'' {\em IEEE Transactions on Automatic
  Control}, vol.~17, no.~1, pp.~15--22, 1972.

\bibitem{ho1980team}
Y.-C. Ho, ``Team decision theory and information structures,'' {\em Proceedings
  of the IEEE}, vol.~68, no.~6, pp.~644--654, 1980.

\bibitem{grover2013approximately}
P.~Grover, S.~Y. Park, and A.~Sahai, ``Approximately optimal solutions to the
  finite-dimensional {Witsenhausen} counterexample,'' {\em IEEE Transactions on
  Automatic Control}, vol.~58, no.~9, pp.~2189--2204, 2013.

\bibitem{grover2015information}
P.~Grover, A.~B. Wagner, and A.~Sahai, ``Information embedding and the triple
  role of control,'' {\em IEEE Transactions on Information Theory}, vol.~61,
  no.~4, pp.~1539--1549, 2015.

\bibitem{rotkowitz2006characterization}
M.~Rotkowitz and S.~Lall, ``A characterization of convex problems in
  decentralized control,'' {\em Automatic Control, IEEE Transactions on},
  vol.~51, no.~2, pp.~274--286, 2006.

\bibitem{nayyar2013decentralized}
A.~Nayyar, A.~Mahajan, and D.~Teneketzis, ``Decentralized stochastic control
  with partial history sharing: A common information approach,'' {\em IEEE
  Transactions on Automatic Control}, vol.~58, no.~7, pp.~1644--1658, 2013.

\bibitem{yuksel2013stochastic}
S.~Y{\"u}ksel and T.~Ba{\c{s}}ar, ``Stochastic networked control systems,''
  {\em AMC}, vol.~10, p.~12, 2013.

\bibitem{swigart2014optimal}
J.~Swigart and S.~Lall, ``Optimal controller synthesis for decentralized
  systems over graphs via spectral factorization,'' {\em IEEE Transactions on
  Automatic Control}, vol.~59, no.~9, pp.~2311--2323, 2014.

\bibitem{lessard2015optimal}
L.~Lessard and S.~Lall, ``Optimal control of two-player systems with output
  feedback. to appear,'' {\em IEEE Transactions on Automatic Control}, vol.~60,
  no.~8, pp.~2129 -- 2144, 2015.

\bibitem{gattami2012robust}
A.~Gattami, B.~M. Bernhardsson, and A.~Rantzer, ``Robust team decision
  theory,'' {\em Automatic Control, IEEE Transactions on}, vol.~57, no.~3,
  pp.~794--798, 2012.

\bibitem{bacsar2008h}
T.~Ba{\c{s}}ar and P.~Bernhard, {\em H-infinity optimal control and related
  minimax design problems: a dynamic game approach}.
\newblock Springer Science \& Business Media, 1990.

\bibitem{colombino2015quadratic}
M.~Colombino, T.~Summers, and R.~Smith, ``Quadratic two-team games,'' in {\em
  IEEE Conference on Decision and Control, Osaka, Japan}, pp.~3784--3789, 2015.

\bibitem{hassibi1999indefinite}
B.~Hassibi, A.~H. Sayed, and T.~Kailath, ``Indefinite-quadratic estimation and
  control,'' {\em Studies in Applied and Numerical Mathematics}, 1999.

\bibitem{basar1978decentralized}
T.~Basar, ``Decentralized multicriteria optimization of linear stochastic
  systems,'' {\em Automatic Control, IEEE Transactions on}, vol.~23, no.~2,
  pp.~233--243, 1978.

\bibitem{swigart2010explicit}
J.~Swigart and S.~Lall, ``An explicit state-space solution for a decentralized
  two-player optimal linear-quadratic regulator,'' in {\em Proceedings of the
  2010 American Control Conference}, pp.~6385--6390, IEEE, 2010.

\end{thebibliography}
\bibliographystyle{ieeetr}

\end{document}